\documentclass[12pt]{amsart}

\usepackage{amssymb}
\usepackage{amsmath}
\usepackage{graphicx}
\allowdisplaybreaks[4]

\theoremstyle{remark}

\newtheorem{para}{\bf}[subsection]

\newtheorem{rem}[para]{\bf Remark}
\theoremstyle{definition}

\newtheorem{dfn}[para]{Definition}

\theoremstyle{plain}

\newtheorem{thm}[para]{Theorem}
\newtheorem{lem}[para]{Lemma}

\newtheorem{prop}[para]{Proposition}

\newenvironment{numequation}{\addtocounter{para}{1}
\begin{equation}}{\end{equation}}

\newcommand{\vep}{\varepsilon}
\newcommand{\vpi}{\varpi}

\newcommand{\bbF}{{\mathbb F}}
\newcommand{\bbG}{{\mathbb G}}
\newcommand{\bbH}{{\mathbb H}}

\newcommand{\bbP}{{\mathbb P}}

\newcommand{\bbT}{{\mathbb T}}

\newcommand{\frg}{{\mathfrak g}}
\newcommand{\frh}{{\mathfrak h}}

\newcommand{\fro}{{\mathfrak o}}

\newcommand{\frx}{{\mathfrak x}}
\newcommand{\fry}{{\mathfrak y}}
\newcommand{\frz}{{\mathfrak z}}

\newcommand{\frX}{{\mathfrak X}}

\newcommand{\cA}{{\mathcal A}}

\newcommand{\cO}{{\mathcal O}}

\newcommand{\cV}{{\mathcal V}}

\newcommand{\Z}{{\mathbb Z}}

\newcommand{\Qp}{{\mathbb Q_p}}
\newcommand{\Cp}{{\mathbb C_p}}

\newcommand{\Fpbar}{\overline{\bbF}_p}
\newcommand{\Fq}{{\mathbb F_q}}
\newcommand{\Fqbar}{\overline{\bbF}_q}

\newcommand{\fronr}{{\hat{\fro}^{nr}}}

\newcommand{\froDx}{\fro_D^*}

\newcommand{\Spec}{{\rm Spec}}
\newcommand{\Spf}{{\rm Spf}}

\newcommand{\lra}{\longrightarrow}

\newcommand{\midc}{{\hskip4pt | \hskip4pt}}

\newcommand{\sub}{\subset}

\begin{document}


\title{On the Analyticity of the group action on the Lubin-Tate space}



\author{Chi Yu Lo}
\address{Indiana University, Department of Mathematics, Rawles Hall, Bloomington, IN 47405, U.S.A.}
\email{loch@indiana.edu}



\begin{abstract}
In this paper we study the analyticity of the group action of the automorphism group $G$ of a formal module $\bar{F}$ of height 2 (defined over $\Fqbar$) on the Lubin-Tate deformation space $X$ of $\bar{F}$. It is shown that a wide open congruence group of level zero attached to a non-split torus acts analytically on a particular disc in $X$ on which the period morphism is not injective. For certain other discs with larger radii (defined in terms of quasi-canonical liftings) we find wide open rigid analytic groups which act analytically on these discs.
\end{abstract}

\maketitle


\tableofcontents



\section{Introduction}\label{intro}

\setcounter{subsection}{1}

{\it The deformation space.} Let $K$ be a finite extension of $\mathbb{Q}_p$ with ring of integers $\fro = \fro_K$, uniformizer $\pi$, and residue field $\Fq$. Denote by $\bar{F}$ a formal $\fro_K$-module of $K$-height 2 over $\Fpbar$. It is well known that $G = {\rm Aut}_A(\bar{F})$ is isomorphic to the group of units $\froDx$ of the maximal compact subring $\fro_D$ of a quaternion division algebra $D$ with center $K$, cf. \cite[1.7]{Drinfeld74}. Therefore, $G$ carries the structure of a locally $K$-analytic group. The deformation space $\frX$ of $\bar{F}$ is (non-canonically) isomorphic to $\Spf(\fronr[[u]])$, where $\fronr$ is the competion of the maximal unramified extension of $\fro_K$and the group $G$ acts naturally on $\frX$ by automorphisms of this formal scheme. In particular, $G$ acts on the associated rigid-analytic space $X = \frX^{\rm rig}$ which we identify (using the chosen coordinate $u$) with the wide open unit disc $\{u \midc |u|<1\}$.

\vskip8pt

{\it Motivation: locally analytic representations.} The motivation for this paper comes from the theory of locally analytic representations of $p$-adic groups. Suppose $\cV$ is a $G$-equivariant vector bundle on $X$. The space of global sections $H^0(X,\cV)$ is then a nuclear Fr\'echet space, and its topological dual space $H^0(X,\cV)'_b$, equipped with the strong topology, is a compact inductive limit of Banach spaces. This space carries a $G$-action, and the question arises if this representation is locally analytic, and what other properties it may have. 

\vskip8pt

For instance, when $\cV = \cO_X$ is the structure sheaf, then the Gross-Hopkins period morphism

$$\Phi: X \lra (\bbP^1)^{\rm rig} \;, \;\; u \mapsto [\phi_0(u):\phi_1(u)] \;,$$

\vskip8pt

cf. \cite{GH}, can be used to show that $H^0(X,\cO_X)'_b$ is indeed a locally analytic representation, cf. section \ref{larger_radii}.\footnote{That $H^0(X,\cO_X)'_b$ is a locally analytic $G$-representation has been shown for more general deformation spaces $X$ of $p$-divisible formal groups and their automorphism groups $G$ by J. Kohlhaase, cf. \cite{KohlhaaseDef}.} That this action is locally analytic is in fact not very difficult to see in our given situation. However, in order to get a better understanding of $H^0(X,\cO_X)'_b$ as a locally analytic representation, we are interested in the subspaces of vectors which are {\it analytic} for certain wide open rigid-analytic groups $\bbG_s^\circ$. In doing so we are following the point of view on locally analytic representations developed by M. Emerton in \cite{Emerton}. We are now going to introduce the groups $\bbG_s^\circ$.

\vskip8pt

{\it The groups $\bbG_s^\circ$.} Let $K_2/K$ be the unramified quadratic extension, and write $\alpha \mapsto \bar{\alpha}$ be the non-trivial Galois automorphism of $K_2$ over $K$. Then we can represent $D$ as a $K$-subalgebra of $M_2(K_2)$ as follows:

\begin{numequation}\label{D} D = \left\{\left(\begin{array}{cc} \alpha & \pi\bar{\beta}  \\ \beta & \bar{\alpha} \end{array} \right) \; \Big| \;\; \alpha, \beta \in K_2 \; \right\} \;.
\end{numequation}

We let $\bbG$ be the algebraic group scheme over $\Spec(\fro_K)$ defined by $\fro_D^*$, i.e., for every unital commutative $\fro_K$-algebra $R$ one has

$$\bbG(R) = \Big(\fro_D \otimes_{\fro_K} R \Big)^* \;.$$

\vskip8pt

Let $\zeta \in K_2^*$ be such that $\bar{\zeta} = -\zeta$, so that $\zeta^2$ is in $\fro_K^*$. Let $a_1, a_2, b_1, b_2$ be indeterminates and put $\Delta =  a_1^2-\zeta^2a_2^2-\vpi(b_1^1-\zeta^2b_2^2)$. Then

$$\bbG = \Spec\left(\fro_K[a_1,a_2,b_1,b_2]\Big[\frac{1}{\Delta}\Big]\right) \;,$$

\vskip8pt

where the co-multiplication is given by

$$\begin{array}{rcl}
a_1 & \mapsto & a_1 a_1' + \zeta^2a_2 a_2' + \vpi b_1b_1' - \zeta^2 \vpi b_2 b_2' \;,\\
&&\\
a_2 & \mapsto & a_2 a_1' + a_1 a_2' + \vpi b_1 b_2' - \vpi b_2 b_1' \;,\\
&&\\
b_1 & \mapsto & a_1 b_1' - \zeta^2 a_2 b_2' + b_1 a_1' + \zeta^2 b_2 a_2' \;,\\
&&\\
b_2 & \mapsto & a_1 b_2' - a_2 b_1' + b_1 a_2' + b_2 a_1' \;.\\
\end{array}$$

\vskip8pt

Let $\bbG_K$ be the base change from $\fro_K$ to $K$, and let $\bbG_K^{\rm rig}$ be the associated rigid-analytic group. Its group of $K$-valued points is equal to $D^*$. For an integer $s \ge 0$ there is a ``wide open'' rigid analytic group $\bbG_s^\circ \sub \bbG_K^{\rm rig}$ whose group of $\Cp$-valued points is given by 

$$\left\{(a_1,a_2,b_1,b_2) \in \bbG_K^{\rm rig}(\Cp) \midc |a_1 -1| < |\pi|^s \,, \,|a_2| < |\pi|^s \,, \, |b_1| < |\pi|^{s-\frac{1}{q+1}} \,, \, |b_2| < |\pi|^{s-\frac{1}{q+1}} \, \right\} \;.$$

\vskip8pt

{\it Critical radii and critical discs.} In \cite{GH} the homogeneous coordinates on $\bbP^1$ are chosen in such a way that the moduli of quasi-canonical lifts which carry an action of an open subgroup of $\fro_{K_2}^*$ are mapped to the points $[1:0]$ and $[0:1]$, which are the fixed points of the non-split torus 

$$\fro_{K_2}^* \; \simeq \; \left\{\left(\begin{array}{cc} \alpha & 0 \\ 0 & \bar{\alpha} \end{array} \right) \; \Big| \;\; \alpha \in \fro_{K_2}^* \; \right\} \; \sub \; \froDx \;.$$

\vskip8pt

With respect to the coordinate $u$ used in \cite{GH}, the absolute values of the moduli of these quasi-canonical lifts (which carry an action of an open subgroup of $\fro_{K_2}^*$) are given by $|u|=0$ and 

$$|u| = |\pi|^{\frac{1}{(q+1)q^s}} \;, \;\; s = 0, 1, \ldots \;.$$

\vskip8pt

For $s \in \Z_{\ge 0}$ we call $r_s = |\pi|^{\frac{1}{(q+1)q^s}}$ a {\it critical radius} and consider the affinoid subdomain

$$\Delta_s = \Big\{u \in X \midc |u| \le |\pi|^{\frac{1}{(q+1)q^s}} \; \Big\} \sub X \;,$$

\vskip8pt

which we call a {\it critical disc}. It is easy to see that the action of $G = \froDx$ on $X$ stabilizes any of the discs $\Delta_s$. Our investigations seem to indicate that the action of $G$ on $\Delta_s$ extends to a rigid-analytic action of $\bbG_s^\circ$ on $\Delta_s$. While, at the moment, we fall short of proving this, we have obtained some partial results in this direction.

\vskip8pt

{\it The results of this paper.} Let 

$$\bbT = \Spec\left(\fro[a_1,a_2]\left[\frac{1}{a_1^2-\zeta^2a_2^2}\right]\right) \sub \bbG$$

\vskip8pt

be the subgroup scheme which corresponds to the unramified torus $\fro_{K_2}^* \sub \froDx$. In section 2, we will show that the action of $\fro_{K_2}^*$ on $\Delta_0$ extends to an analytic action of the rigid-analytic subgroup 

\begin{numequation}\label{non_split_torus} \begin{array}{rcl} \bbT_0^\circ & = & \left\{(a_1,a_2,b_1,b_2) \in \bbG_K^{\rm rig} \midc |a_1 -1| < 1 \,, \,|a_2| < 1 \,, \, b_1 = b_2 = 0 \, \right\} \\
&&\\
& = & (\bbT_K)^{\rm rig} \cap \bbG_0^\circ 
\end{array}
\end{numequation}

\vskip8pt

on $\Delta_0$, cf. theorem \ref{main1} (2). We prove this by explicitly analyzing the group action of $\fro_{K_2}^*$. For $g=\left(
\begin{array}{cc}
	\alpha &  0\\
	0 & \bar{\alpha} \\
\end{array}
\right)$, with $\alpha \in \fro_{K_2}^*$, we write

$$g.u = \sum_{n=0}^\infty a_n(g)u^n \;.$$

\vskip8pt

In section \ref{power_series} we show that each function $a_n(g)$ is a polynomial in $E = \frac{\bar{\alpha}}{\alpha}$, and that $a_n$ vanishes identically if $n$ is not of the form $1+k(q+1)$ for $k \in \Z_{\ge 0}$. Put $b_k(E) = a_{1+k(q+1)}(E)$. Then, in section \ref{rational_functions}, we show that $b_k(E) = \frac{1}{\pi^k}EQ_k(E^{q+1})$, where $Q_k(x)$ is a polynomial with coefficients in $\fro_K$, and $\deg(Q_k) \le k$. The key problem is then to estimate $|Q_k(x)|$ when $|x-1| \le r$ for $r < 1$.  This requires some fairly delicate arguments which are quite elaborate. 

\vskip8pt

In section \ref{larger_radii} we analyze the group action via the derived action of its Lie algebra. In this section we assume eventually that $K = \Qp$. For every disc $\Delta_s$ we show that a certain rigid-analytic subgroup $\bbH_s^\circ$ of $\bbG_K^{\rm rig}$ acts analytically on $\Delta_s$, cf. theorem \ref{main2} for details. However, $\bbH_s^\circ$ is always strictly contained in the analytic group $\bbG_s^\circ$ defined above.

\section{Analyticity of the non-split torus on the first critical disc}

\subsection{The power series describing the group action}\label{power_series}

\begin{para}
According to \cite[\S 25]{GH}, the period map $\Phi(u) = [\phi_0(u):\phi_1(u)]$ from the deformation space $X=\{u:|u|<1\}$ to the rigid-analytic projective space $(\mathbb{P}^1)^{\rm rig}$ can be described by power series 

$$\phi_0(u)=\sum_{n=0}^\infty c_n u^n \;, \hskip16pt \phi_1(u) = \sum_{n=1}^\infty d_n u^n \;,$$ 

\vskip8pt

whose coefficients are given as follows

\[c_n=\left\{
\begin{array}{lcl}
	1 & & \textrm{if } n=0 \\
	\\
	\pi^{\frac{-k-1}{2}} & & \textrm{if } n \textrm{ is of the form } q^{2a_0}+q^{2a_1+1}+\dots+q^{2a_k+k}\\
	& & \textrm{ with } 0\leq a_0\leq \dots \leq a_k \textrm{ for some odd k} \\
	\\
	0 & & \textrm{otherwise}\\
\end{array}
\right.\]

\vskip8pt

and

\[d_n=\left\{
\begin{array}{lcl}
	\pi^{\frac{-k}{2}} & & \textrm{if } n \textrm{ is of the form } q^{2a_0}+q^{2a_1+1}+\dots+q^{2a_k+k}\\
	& & \textrm{ with } 0\leq a_0\leq \dots \leq a_k \textrm{ for some even k} \\
	\\
	0 & & \textrm{otherwise}\\
\end{array}
\right. \;.\]

\vskip8pt

In particular, $c_n=0$ if $q+1 \nmid n$ and $d_n=0$ if $q+1 \nmid n-1$.

\vskip8pt

The group $G$ acts on $\bbP^1$ by linear transformations. If $g=\left(
\begin{array}{cc}
	\alpha &  \pi\bar{\beta}\\
	\beta & \bar{\alpha} \\
\end{array}
\right)\in G$ and $[x_0:x_1] \in \bbP^1$, then 

\begin{numequation}\label{action} g\cdot [x_0:x_1]=[\alpha x_0+\beta x_1 : \pi\bar{\beta} x_0+\bar{\alpha} x_1] \;,
\end{numequation}

\vskip8pt

cf. \cite[25.13]{GH}. In this section we will only be interested in the action of the subgroup

$$\left\{\left(\begin{array}{cc} \alpha & 0 \\ 0 & \bar{\alpha} \end{array} \right) \; \Big| \;\; \alpha \in \fro_{K_2}^* \; \right\} \; \sub \; \froDx \;,$$

\vskip8pt

cf. \ref{D}. Let $g=\left(
\begin{array}{cc}
	\alpha & 0 \\
	0 & \bar{\alpha} \\
\end{array}
\right)$ be an element of this subgroup, with $\alpha\in \fro_{K_2}^*$. Then the action of $g$ on $X$ is given by a power series

\[g\cdot u=\sum_{n=0}^\infty a_n(g) u^n\; .\]
\vskip8pt

In order to understand how $a_n(g)$ depends on $g$ we make use of the period map $\Phi$. Since this map is $G$-equivariant, we have

\[ [\phi_0(g\cdot u),\phi_1(g\cdot u)]=g\cdot[\phi_0(u),\phi_1(u)]=[\alpha\phi_0(u),\bar{\alpha}\phi_1(u)] \; . \]
\vskip8pt

And hence

\[ E\phi_1(u)\phi_0(g\cdot u)= \phi_0(u)\phi_1(g\cdot u)   \]
\vskip8pt

where $E:=\frac{\bar{\alpha}}{\alpha}$. By comparing the coefficients of $u^n$ in the above equation, we get

\begin{numequation}
\label{eq1}
	E\sum_{l\leq n} d_l c_m \sum_{\substack{\sum r_k=m\\ \sum k r_k=n-l}} \prod_k a_k^{r_k}=\sum_{m\leq n} c_m d_l \sum_{\substack{\sum r_k=l \\ \sum k r_k=n-m}}\prod_k a_k^{r_k}  \;.
\end{numequation}

\vskip8pt

By induction, we can see that the function $a_n$ is actually a function of $E$. So instead of writing $a_n(g)$, we write $a_n(E)$ from now on. When $n=0$, equation \ref{eq1} becomes

\[ 0= \phi_1(a_0)\;.\]
\vskip8pt
Hence $a_0(E)=0$, since $|a_0(E)|\leq |\pi|$ and $\phi_1$ is injective on $\{u:|u|\leq |\pi|^{\frac{1}{q+1}}\}$.\vskip8pt

When $n=1$, the equation \ref{eq1} becomes

\[ Ed_1c_0= c_0 d_1 a_1,\]

\vskip8pt

hence $a_1(E)=E$. For $2\leq n \leq q$, we have

\[ 0=c_0d_1 a_n \;,\]
\vskip8pt

and thus $a_n(E)=0$.

\begin{lem}
If $q+1  \nmid \,\, n-1$, then $a_n(E)=0$.
\end{lem}
\begin{proof}
We will prove by induction. Suppose $a_n(E)=0$ for $n\leq N-1$ and $q+1 \nmid n-1$. Want to check the case $n=N$.\vskip8pt

If $q+1 |N-1$, then we have nothing to show. Now we assume $q+1 \nmid N-1$. By the induction hypotheses, left hand side of equation \ref{eq1} becomes

\[E\sum_{l\leq N} d_l c_m \sum_{\substack{\sum r_{1+k(q+1)}=m \\ \sum (1+k(q+1)) r_{1+k(q+1)}=N-l}}\prod_k a_{1+k(q+1)}^{r_{1+k(q+1)}}
=  E\sum_{l\leq N} d_l c_m \sum_{\substack{\sum r_{1+k(q+1)}=m \\  \sum k\, r_{1+k(q+1)}=\frac{N-m-l}{q+1} }} \prod_k a_{1+k(q+1)}^{r_{1+k(q+1)}}
\; .\]

\vskip8pt

Since $q+1 | m+l-1$ if $c_m d_l$ does not vanish, left hand side become zero as $\frac{N-m-l}{q+1}$ is not an integer. Similarly, the right hand side becomes

\[
c_0d_1a_{N}\; .\]
\vskip8pt

Hence $a_{N}(E)=0$ follows. \end{proof}

\vskip8pt

From now on, we can rewrite

\[g\cdot u=\sum b_{k}(E)u^{1+k(q+1)}\]
\vskip8pt

 with $b_n(E):=a_{1+n(q+1)}(E)$ and rewrite equation \ref{eq1} as

\[	E\sum_{l\leq n} d_l c_m \sum_{\substack{\sum r_k=m\\ \sum k r_k=\frac{n-m-l}{q+1}}}\prod_k b_{k}^{r_k}=\sum_{m\leq n} c_m d_l \sum_{\substack{\sum r_k=l\\ \sum k r_k=\frac{n-m-l}{q+1}}}\prod_k b_{k}^{r_k}  \]
\vskip8pt

or

\begin{numequation}
\label{eq2}
b_n=E\sum_{l\leq n} d_l c_m \sum_{\substack{\sum r_k=m\\ \sum k r_k=\frac{n-m-l}{q+1}}}\prod_k b_{k}^{r_k}-\sum_{\substack{m\leq n\\ m+l>1}} c_m d_l \sum_{\substack{\sum r_k=l\\ \sum k r_k=\frac{n-m-l}{q+1}}}\prod_k b_{k}^{r_k}
\end{numequation}
\vskip8pt

Let us consider the first few terms when $n \leq 4$:



\[\begin{aligned}
b_1(E) & =\frac{1}{\pi}E\binom{q+1}{0} b_0^{q+1}-\frac{1}{\pi}b_0 = \frac{1}{\pi}E\left(E^{q+1}-1\right) \;,\\
\\
b_2(E) & =\frac{1}{\pi}E\binom{q+1}{1} b_0^{q}b_1-\frac{1}{\pi}b_1 = \frac{1}{\pi^2}(q+1)(E^{q+1}-1)^2+\frac{1}{\pi^2}q(E^{q+1}-1) \;,\\
\\
b_3(E) & =\frac{1}{\pi^3}\frac{(q+1)(3q+2)}{2}E(E^{q+1}-1)^3+\frac{1}{\pi^3}\frac{5q(q+1)}{2}E(E^{q+1}-1)^2 \\
 &\hskip300pt +\frac{1}{\pi^3}q^2E(E^{q+1}-1) \;, \\
\\
b_4(E) & = \frac{1}{\pi^4}\frac{(q+1)(2q+1)(4q+3)}{3}E(E^{q+1}-1)^4 +\frac{1}{\pi^4}\frac{q(q+1)(37q+26)}{6}E(E^{q+1}-1)^3 \\
 & \hskip150pt +\frac{1}{\pi^4}\frac{9q^2(q+1)}{2}E(E^{q+1}-1)^2+\frac{1}{\pi^4}q^3E(E^{q+1}-1) \;.\\
\end{aligned}\]

\vskip12pt

We remark that J. Kohlhaase has computed the functions $b_n$ for $n=0,1,2$, cf. \cite[Thm. 1.19]{Koh} (what is denoted by $\alpha_1$ in loc. cit. coincides with what is here denoted by $E$).\vskip8pt
\end{para}

\vskip12pt

\subsection{The coefficients as rational functions on the torus}\label{rational_functions}

Here we will present some results about the terms $b_k$ or $a_{1+k(q+1)}$ as functions of $E=\frac{\bar{\alpha}}{\alpha}$.\vskip8pt

\begin{lem}
\begin{enumerate}
	\item $b_{k}(E)$ is of the form $\pi^{-k}E Q_k(E^{q+1})$ where $Q_k\in \fro_K[x]$\;.
	\vskip8pt
 	\item With $Q_k$ as in (1) we have $\deg_x(Q_k) \leq k$\;.
	\vskip8pt
	\item With $Q_k$ as in (1) we have $Q_k(0)\equiv (-1)^{k} \mod \pi$\;.
	\vskip8pt
	\item $||b_{k}||=|\pi^{-k}|$ where the supremum norm is taken over $|E|\leq 1$\;.
\end{enumerate}

\end{lem}

\begin{proof}
Part(4) follows from part (1) and part (3).

\vskip8pt

We are now going to prove (1), (2) and (3) at the same time by induction on $k$. It is clear that $b_0$ and $b_1$ satisfy the statements.

\vskip8pt

Suppose the statements are true for $k\leq N-1$, where $N\geq 2$. Then equation \ref{eq2} can be rewrite as
\begin{numequation}
\label{eq3}
b_{N}=C_1+C_2+C_3+C_4\;,
\end{numequation}
where

\begin{align*}
	C_1 & = -\pi^{-1}b_{N-1} \;\;, \\
	C_2& =  E d_1 c_{q+1}
	\sum_{\substack{|\underline{r}|=q+1\\
	r_k\geq q,\,\, \exists k\geq 0\\
	\sum k r_k=N-1\\
	} } \genfrac{(}{)}{0pt}{}{q+1}{\underline{r}}\underline{b}^{\underline{r}} \;\;,\\
	C_3 & = E d_1 c_{q+1}
	\sum_{\substack{|\underline{r}|=q+1\\
	r_k\leq q-1,\,\, \forall k\geq 0\\
	\sum k r_k=N-1\\
	} } \genfrac{(}{)}{0pt}{}{q+1}{\underline{r}}\underline{b}^{\underline{r}} \;\;,\\
	C_4 & = E\sum_{q+2< m+l} d_m c_{l}
	\sum_{\substack{|\underline{r}|=m\\
	\sum k r_k=N-\frac{m+l-1}{q+1}\\
	} } \binom{m}{\underline{r}}\underline{b}^{\underline{r}} -  \sum_{q+2< m+l} d_m c_{l}
	\sum_{\substack{|\underline{r}|=l\\
	\sum k r_k=N-\frac{m+l-1}{q+1}\\
	} } \binom{l}{\underline{r}}\underline{b}^{\underline{r}} \;\;.\\
\end{align*}

In the above expression, $\underline{r}$ denotes the multi-index $(r_0,r_1,r_2,...)$ and $|\underline{r}|$ denotes $\sum r_k$. If $n=|\underline{r}|$, $\binom{n}{\underline{r}}$ denotes $\frac{n!}{r_0!r_1!r_2!...}$. Finally$, \underline{b}^{\underline{r}}$ denotes $\prod b_{k}^{r_k}$.\vskip8pt

Part (1) and part (2) follows from directly from the induction hypothesis.

\vskip8pt

To prove part(3), multiply $b_{N}$ by $\pi^N$ and modulo $\pi$. In particular, $ \pi^N C_3\equiv \pi^N C_4\equiv  0$ as $\frac{m+l-1}{q+1}+\nu(c_md_l)>0$ if $m+l>q+2$. Hence $\pi^N \frac{b_{N}}{E}\equiv -\pi^{N-1}\frac{b_{N-1}}{E}+\pi^N\frac{C_2}{E} \mod \pi$. Put $E=0$ and the result follows. \end{proof}

The goal of this section is to find estimates for $b_n$ when $|E-1|< 1$. To obtain these estimates we need to describe $b_n$ or $Q_n$ more precisely. We will use the recursive formula $\ref{eq2}$ to define polynomials $b_{n,k}(E)$ which are of the form  $\pi^{-k} EQ_{n,k}(E^{q+1})$ with $Q_{n,k}(x) \in \fro_K[x]$ for $0 \leq k\leq n$ such that $b_n = \sum b_{n,k}$ and with good control on the order of $(x-1)$ in $Q_{n,k}$. In particular, $\pi^{-n}Q_{n}=\sum \pi^{-k}Q_{n,k}$ and $||b_{n,k}|| \leq |\pi|^{-k}$.

\vskip8pt

First of all, $b_{0,0}(E):=b_{0}(E)=E$. Suppose we have already defined $b_{n,k}$ for $0\leq k\leq n<N$. Then $\ref{eq3}$ suggests the following definition for $s<N$:

\begin{align}\label{eq4}
b_{N,s}= & \sum_{\substack{m+l>1\\ m<l}}c_m d_l \left(E-E^{q^{\left\lfloor \log_q(l)\right\rfloor}}\right)\sum_{\substack{|\underline{r,i}|=m \\ \sum_{k,i} kr_{k,i}=N-\frac{m+l-1}{q+1}\\ \sum_{k,i}ir_{k,i}=s+\nu(c_md_l) }}\binom{m}{\underline{r,i}}\underline{b,i}^{\underline{r,i}}\\
 & +\sum_{\substack{m+l>1\\ m<l}}c_m d_l E^{q^{\left\lfloor \log_q(l)\right\rfloor}}\sum_{\substack{|\underline{r,i}|=m \\ \sum_{k,i} kr_{k,i}=N-\frac{m+l-1}{q+1}\\ \sum_{k,i}ir_{k,i}=s+\nu(c_md_l)+1  }}\left(\binom{m}{\underline{r,i}}-\binom{m+q^{\left\lfloor \log_q(l)\right\rfloor}}{\underline{r,i}+q^{\left\lfloor \log_q(l)\right\rfloor}}\right)\underline{b,i}^{\underline{r,i}}\notag\\
 & - \sum_{\substack{m+l>1\\ m<l}}c_m d_l  \sum_{\substack{|\underline{r,i}|=l,\,\, r_{0,0}<q^{\left\lfloor \log_q(l)\right\rfloor},\,\, p\nmid\binom{l}{\underline{r,i}}  \\ \sum_{k,i} kr_{k,i}=N-\frac{m+l-1}{q+1}\\ \sum_{k,i}ir_{k,i}=s+\nu(c_md_l) }} \binom{l}{\underline{r,i}} \underline{b,i}^{\underline{r,i}}\notag\\
 &-\sum_{\substack{m+l>1\\ m<l}}c_m d_l  \sum_{\substack{|\underline{r,i}|=l,\,\, r_{0,0}<q^{\left\lfloor \log_q(l)\right\rfloor},\,\, p|\binom{l}{\underline{r,i}}  \\ \sum_{k,i} kr_{k,i}=N-\frac{m+l-1}{q+1}\\ \sum_{k,i}ir_{k,i}=s+\nu(c_md_l)+1 }} \binom{l}{\underline{r,i}} \underline{b,i}^{\underline{r,i}}\notag\\
 & + \sum_{\substack{m+l>1\\ l<m}}c_m d_l \left(E^{1+q^{\left\lfloor \log_q(m)\right\rfloor}}-1\right)\sum_{\substack{|\underline{r,i}|=l  \\   \sum_{k,i} kr_{k,i}=N-\frac{m+l-1}{q+1}\\ \sum_{k,i}ir_{k,i}=s+\nu(c_md_l)  }}\binom{l}{\underline{r,i}}\underline{b,i}^{\underline{r,i}}\notag\\
 & +\sum_{\substack{m+l>1\\ l<m}}c_m d_l E^{1+q^{\left\lfloor \log_q(m)\right\rfloor}}\sum_{\substack{|\underline{r,i}|=l   \\ \sum_{k,i} kr_{k,i}=N-\frac{m+l-1}{q+1}\\ \sum_{k,i}ir_{k,i}=s+\nu(c_md_l)+1\\  }}\left(\binom{l+q^{\left\lfloor \log_q(m)\right\rfloor}}{\underline{r,i}+q^{\left\lfloor \log_q(m)\right\rfloor}}-\binom{l}{\underline{r,i}}\right)\underline{b,i}^{\underline{r,i}} \notag\\
 & +E\sum_{\substack{m+l>1\\ l<m}}c_m d_l  \sum_{\substack{|\underline{r,i}|=m,\,\, r_{0,0}<q^{\left\lfloor \log_q(m)\right\rfloor},\,\, p\nmid\binom{m}{\underline{r,i}}  \\ \sum_{k,i} kr_{k,i}=N-\frac{m+l-1}{q+1}\\ \sum_{k,i}ir_{k,i}=s+\nu(c_md_l) }} \binom{m}{\underline{r,i}} \underline{b,i}^{\underline{r,i}}\notag\\
& +E\sum_{\substack{m+l>1\\ l<m}}c_m d_l  \sum_{\substack{|\underline{r,i}|=m,\,\, r_{0,0}<q^{\left\lfloor \log_q(m)\right\rfloor},\,\, p|\binom{m}{\underline{r,i}}  \\ \sum_{k,i} kr_{k,i}=N-\frac{m+l-1}{q+1}\\ \sum_{k,i}ir_{k,i}=s+\nu(c_md_l)+1 }} \binom{m}{\underline{r,i}} \underline{b,i}^{\underline{r,i}}\;.\notag
\end{align}

And $b_{N,0}:=b_N-\sum_{s>0}b_{N,s}$.\vskip8pt

In the above expression, $\underline{r,i}$ denote the multi-index $\{r_{k,i},  0\leq i\leq k \}$. The terms $|\underline{r,i}|:=\sum_{0\leq i\leq k}r_{k,i}$, $\binom{n}{\underline{r,i}}:=\frac{n!}{\prod_{0\leq i \leq k}r_{k,i}!}$ and $\underline{b,i}^{\underline{r,i}}:=\prod_{0\leq i\leq k}b_{k,i}^{r_{k,i}}$ are defined similarly as before. Moreover, if $n< q^t$, then $\binom{n+q^t}{\underline{r,i}+q^t}:=\frac{\left(n+q^t\right)!}{(r_{0,0}+q^t)!\prod_{0\leq i \leq k,k>0}r_{k,i}!}$\vskip8pt

And $Q_{n,k}(x)$ is defined s.t.

\[EQ_{n,k}(E^{q+1})=\pi^k b_{n,k}(E)\; .\]

\vskip8pt

\begin{rem}
\begin{enumerate}
	\item Observe that when $l>m$, $c_m d_l\neq 0$ iff $c_{l-q^{\left\lfloor \log_q(l)\right\rfloor}}d_{m+q^{\left\lfloor \log_q(l)\right\rfloor}}\neq 0$. When $m>l$, $c_m d_l\neq 0$ iff $c_{l+q^{\left\lfloor \log_q(m)\right\rfloor}}d_{m-q^{\left\lfloor \log_q(m)\right\rfloor}}\neq 0$
	\item Suppose $|\underline{r,i}|=n$. If $q^t>n $, then $p|\binom{n}{\underline{r,i}}-\binom{n+q^{t}}{\underline{r,i}+q^t}$. We can prove $p|\binom{n+q^t}{r+q^t}-\binom{n}{r}$ first:\vskip8pt
	For any $r< k\leq n$, $ord_p(q^t+k)=ord_p(k)$ and $\frac{q^t+k}{p^{ord_p(k)}}\equiv \frac{k}{p^{ord_p(k)}} \mod p$ , so $p|\binom{n+q^t}{r+q^t}$ iff $p|\binom{n}{r}$. If $p\nmid\binom{n}{r}$, $\prod_{r< k\leq n}\frac{k}{p^{ord_p(k)}}\equiv \prod_{r< k\leq n}\frac{q^t+k}{p^{ord_p(k)}}\mod p$ and so the result follows as $\binom{n}{r}=\frac{\prod_{r< k\leq n}\frac{k}{p^{ord_p(k)}}}{\prod_{0<k\leq n-r}\frac{k}{p^{ord_p(k)}}} $ and $\binom{q^t+n}{q^t+r}=\frac{\prod_{r< k\leq n}\frac{q^t+k}{p^{ord_p(k)}}}{\prod_{0<k\leq n-r}\frac{k}{p^{ord_p(k)}}} $.\vskip8pt
	And the general case follows as $\binom{n}{\underline{r,i}}-\binom{n+q^{t}}{\underline{r,i}+q^t}=\left(\binom{n}{r_{0,0}}-\binom{n+q^{t}}{r_{0,0}+q^t}\right)\frac{(n-r_{0,0})!}{\prod_{0\leq i\leq k,k>0}r_{k,i}!}$.
	\item $E^{q+1}-1|E-E^{q^{\left\lfloor \log_q(l)\right\rfloor}}$ when $d_l\neq 0$ as $\left\lfloor \log_q(l)\right\rfloor$ is even for such $l$. $E^{q+1}-1|E^{1+q^{\left\lfloor \log_q(m)\right\rfloor}}$ when $c_m\neq 0$ and $m>0$ as $\left\lfloor \log_q(m)\right\rfloor$ is odd for such $m$.
\end{enumerate}
\end{rem}

\vskip12pt

\subsection{The function $R$}
To express the sought-for estimates of the functions $b_{n,k}$, we need to define some auxiliary functions and study their basic properties.

\vskip8pt

\begin{dfn} 

\begin{enumerate}
	\item For $r\geq 0$, define $T_r$ by the
	
	\[T_r=1+q+q^2+\dots+q^r\;.\]
	\vskip8pt
	
	\item Suppose $n>0$ is an integer such that $T_r \leq n <T_{r+1}$. For, $l\in \mathbb{Z}_{\geq 0}$, define $n_l$ backward inductively by
	
	\[n_l=\begin{cases} 0, &\text{if }l\geq r+1 \\
	\left\lfloor \frac{n-n_rT_r-n_{r-1}T_{r-1}-\dots-n_{l+1}T_{l+1}}{T_l}\right\rfloor, &\text{if }0\leq l \leq r\;.\\
	\end{cases}\]
	
		\vskip8pt
		
	Define a mapping $\sigma:\mathbb{Z}\rightarrow \oplus_{\mathbb{Z}_{\geq 0}}\mathbb{Z}_{\geq 0}$ by
	
	\[ \sigma(n)=\begin{cases}
	\{n_l\}_{l\in \mathbb{Z}_{\geq 0}}, & \text{ if }n \geq 0\\
	\{0\}_{l\in \mathbb{Z}_{\geq 0}},& \text{ if }n<0\;.\\
	\end{cases}\]
	
	\vskip8pt
	\item Define functions	$R',P:\oplus_{\mathbb{Z}_{\geq 0}}\mathbb{Z}_{\geq 0}\rightarrow \mathbb{Z}_{\geq 0}$ by
	
	\[R'\left(\{n_l\}_{l\in \mathbb{Z}_{\geq 0}}\right)  = \sum_l n_lq^l\]
	\vskip8pt
	
	and
	
	\[P\left(\{n_l\}_{l\in \mathbb{Z}_{\geq 0}}\right)  = \sum_l n_lT_l\;.\]
	\vskip8pt
	
	\item Finally, define $R$ by
	
	\[R(n):=R'\left(\sigma(n)\right)\;.\]
	\vskip8pt
	
\end{enumerate}
\end{dfn}

\begin{rem}
\label{rem1}
\begin{enumerate}
	\item In the above definition, the sequences $\{n_l\}$ for some non-negative integer $n$ is characterized by:
	\vskip8pt
		
		\begin{itemize}
			\item $n_l\leq q$ for all $l$
			
			\vskip5pt
			
			\item $n_l=q$ for at most one $l$. If there exist such $l$, $n_k=0$ for all $k<l$.
		\end{itemize}
		
		\vskip8pt
		
	This is because $qT_r+1=T_{r+1}$ and $(q-1)T_r+(q-1)T_{r-1}+\dots+(q-1)T_{s+1}+qT_s = T_{r+1}-(r-s+1)<T_{r+1}$ if $r\geq s$.
	
\vskip8pt
		
	\item $P$ is the left inverse of $\left.\sigma\right|_{\mathbb{Z}_{\geq 0}}$.
	
	\vskip8pt
	
	\item We can give $\oplus_{\mathbb{Z}_{\geq 0}}\mathbb{Z}_{\geq 0}$ a lexicographical order: $\{n_l\}>\{m_l\}$ if there exists $N\geq 0$ s.t. $n_N>m_N$ and $n_l=m_l$ for all $l>N$. Moreover, if $\{n_l\}=\sigma(n)$, then $\{n_l\}\geq \{m_l\}$ for all $\{m_l\}$ s.t. $P(m_l)=n$.\vskip8pt
\end{enumerate}
\end{rem}

\vskip8pt

Here are some basic properties of the map $R$.

\vskip8pt

\begin{lem}
\label{lem3}
\begin{enumerate}
	\item $R(n)$ is non-decreasing and $R(n+1)-R(n)\leq 1$\;.
	\vskip8pt
	\item $R(i+j)\leq R(i)+R(j)$\;.
	\vskip8pt
	\item If $n>0$, then $qR(n)\geq R(qn+1)$\;.
	\vskip8pt
	\item $R(n)\geq \frac{q-1}{q}n$ and the equality only holds at $n=0$\;.
\end{enumerate}
\end{lem}

\begin{proof}
Fix $n\in \mathbb{Z}_{\geq 0}$, Set $\{n_l\}=\sigma(n)$. Suppose $\{N_l\}\in \oplus_{\mathbb{Z}_{\geq 0}}\mathbb{Z}_{\geq 0}$ s.t. $P(\{N_l\})$ also equals to $n$.\vskip8pt
Claim: $R'(\{N_l\})\geq R'(\{n_l\})=R(n)$. Equivalently, $R'$ attains minimum at the biggest element $\{n_l\}$ among $\{N_l\}$ with $P(\{N_l\})=n$.  \vskip8pt
Proof of the claim. Suppose $n<T_{r+1}$. Then for all $\{N_l\}$ with $P(\{N_l\})=n$, $N_l=0$ for all $l\geq r+1$.

\vskip8pt

Case 1) \;\; $\exists l \geq 0 $ s.t. $N_l\geq q$. Choose $i_0$ be the largest such index.

\vskip8pt

Case 1A) \;\; $\exists l<i_0$ s.t. $N_l\neq 0$. Choose $j_0$ be the smallest such index. Define $\tilde{N}_l$ by

\[\tilde{N}_l=\begin{cases}
N_l, &\text{ if }l\neq i_0+1,i_0,j_0,j_0-1\\
N_{i_0+1}+1, &\text{if }l=i_0+1\\
N_{i_0}-q, &\text{if }l=i_0\\
N_{j_0}-1, &\text{if }l=j_0\\
q, &\text{if }l=j_0-1\geq 0\\
\end{cases}\]

\vskip8pt

Then $P(\{\tilde{N}_l\})=n$ and $\{\tilde{N}_l\}>\{N_l\}$ in lexicographical order. Also, $R'(\{\tilde{N}_l\})=R'(\{N_l\})$ if $j_0\geq 1$ and $R'(\{\tilde{N}_l\})=R'(\{N_l\})-1$ if $j_0=0$.

\vskip8pt

Case 1B)\;\; $N_l=0$ for all $l<i_0$ and $N_{i_0}\geq q+1$. Define $\tilde{N}_l$ by

\[\tilde{N}_l=\begin{cases}
N_l, &\text{ if }l\neq i_0+1,i_0,i_0-1\\
N_{i_0+1}+1, &\text{if }l=i_0+1\\
N_{i_0}-q-1, &\text{if }l=i_0\\
q, &\text{if }l=i_0-1 \geq 0\\
\end{cases}\]
\vskip8pt

Again, $P(\{\tilde{N}_l\})=n$ and $\{\tilde{N}_l\}>\{N_l\}$. In this situation, we have $R'(\{\tilde{N}_l\})=R'(\{N_l\})$ if $i_0\geq 1$ and $R'(\{\tilde{N}_l\})=R'(\{N_l\})-1$ if $i_0=0$.

\vskip8pt

Case 1C) \;\; $N_l=0$ for all $l<i_0$ and $N_{i_0}= q$. Then $\{N_l\}=\{n_l\}$ by the remark \ref{rem1}.\vskip8pt
Case 2)\;\; $N_l\leq p-1$ for all $l$. Then $\{N_l\}=\{n_l\}$ as in Case 1C.

\vskip8pt

The conclusion we can draw here is that we can increase the lexicographical order of $\{N_l\}$ successively while the $R'$ value does not increase at the same time until we get $\{n_l\}$. Hence the claim follows.

\vskip8pt

We will first focus on $n,i,j\geq 0$.
Now for part (2) of the Lemma, we define $\{i_l\}:=\sigma(i)$ and $\{j_l\}:=\sigma(j)$. Define $N_l:=i_l+j_l$ for all $l$. Observe that $P(\{N_l\})=i+j$, so $R(i)+R(j)=R'(\{N_l\})\geq R(i+j)$.

\vskip8pt

For part (1) of the lemma, define $N_l$ by:

\[N_l=
\begin{cases}
n_l, &\text{if }l>0\\
n_0+1, &\text{if }l=0
\end{cases}
\]

\vskip8pt

where $\{n_l\}=\sigma(n)$. In particular, $P(\{N_l\})=n+1$. Hence,

\[R(n)+1=R'(\{N_l\})\geq R(n+1) \;.\]

\vskip8pt

Since $\{n_l\}=\sigma$, so either $n_l\leq q-1$ for all $l$, or $n_l\leq q-1$ except $n_i=q$ for some $i\geq 0$ and $n_l=0$ for all $l<i$.(Case 1C and 2 in the proof of previous claim.) In the first case, $\{N_l\}$ is in Case 1C or 2, hence $\{N_l\}=\sigma(n+1)$ and $R(n+1)=R(n)+1$. In the second case, $\{N_l\}$ is in Case 1A or 1B and it is not hard to see $R'(\{N_l\})=R(n+1)+1$ and hence $R(n+1)=R(n)$ in this case. In particular, $R(n+1)\geq R(n)$ in all cases.

\vskip8pt

For part (3), define $N_l$ by

\[N_l=
\begin{cases}
qn_l, &\text{if }l>0\\
qn_0+1, &\text{if }l=0
\end{cases}
\]

\vskip8pt

where $\{n_l\}=\sigma(n)$. In particular, $P(\{N_l\})=qn+1$ and $R'(\{N_l\})=qR(n)+1$. Since $n>0$, $\{n_l\}\neq 0$ and $\{N_l\}$ is in Case 1A or 1B. Hence $qR(n)+1=R'(\{N_l\})>R(qn+1) $ and thus, $qR(n)\geq R(qn+1)$.

\vskip8pt

It is clear that when $n< 0$, $R(n)=0\leq R(n+1)\leq 1$. So part 1 follows. If $i\geq 0> j$, then $i+j<i$ and hence $R(i+j)\leq R(i)=R(i)+R(j)$. Similarly for $i,j<0$. Hence part 2 follows.

\vskip8pt

For part (4), Observe that $T_l(q-1)=q\cdot q^l-1$ for all $l>0$ and hence $(q-1)\sum n_l T_l< q \sum n_l q^l$ unless $\{n_l\}=\{0\}$ in which the equality holds. Therefore $R(n)\geq \frac{q-1}{q}n$ when $n\geq 0$. And it is also clear that the strict inequality holds for $n<0$. \end{proof}

\vskip12pt

\subsection{Estimates for the norms of the coefficients}

We start by estimating the order of vanishing of the polynomial $Q_{n,n}$ at $x=1$.

\begin{prop}
$\text{ord}_{x-1}(Q_{n,n})\geq R(n)$. Furthermore, the equality holds when $n=T_l$ for some $l\geq 0$.  In particular, $Q_n(x)\equiv (x-1)^{R(n)}h(x) \mod \pi$ for some $h\in \mathbb{Z}[\pi][x]$.

\end{prop}

\begin{proof}
We will use induction on $n$. When $s=n>0$, equation \ref{eq4} becomes

$$b_{n,n} \; = \; \frac{E^{q+1}-1}{\pi}b_{n-1,n-1} + \pi^{-1}E \sum_{\substack{|\underline{r,i}|=q+1\\
	r_{k,i}=0 \text{ if }i<k\\
	r_{k,k}\geq q,\,\, \exists k \text{ with } k>0\\
	\sum_{k} k r_{k,k}=n-1\\
		} } \binom{q+1}{\underline{r,i}}\underline{b,i}^{\underline{r,i}}$$

\vskip8pt

as $i\leq k$ and $\nu(c_m)+\nu(d_l)>-\frac{m+l-1}{q+1}$ for $m+l>q+2$.\vskip8pt

Here $\text{ord}_{E^{q+1}-1}\frac{E^{q+1}-1}{\pi}b_{n-1,n-1}\geq 1+R(n-1)\geq R(n)$ by induction.\vskip8pt

Also, $\text{ord}_{E^{q+1}-1}\underline{b,i}^{\underline{r,i}}\geq \sum_k r_{k,k}R(k)\geq R(1+\sum kr_{k,k})=R(n)$ by lemma $\ref{lem3}$(3) as $r_{k,k}\geq q$ for some $k>0$.\vskip8pt

Therefore, we get $\text{ord}_{x-1}(Q_{n,n})\geq R(n)$.\vskip8pt

$b_{T_0,T_0}=b_0=E$ and $b_{T_1,T_1}=b_1=E\frac{E^{q+1}-1}{\pi}$, so $b_{T_l}=R(T_l)$ for $l=0,1$. And we proceed with induction.

\vskip8pt

Suppose now $n=T_{l+1}$ with $l>0$, then

$$b_{n,n} \; = \; \frac{E^{q+1}-1}{\pi}b_{n-1,n-1}+ \pi^{-1}E\binom{q+1}{1}b_{T_{l},T_{l}}^{q}b_0 + \pi^{-1}E \sum_{\substack{|\underline{r,i}|=q+1\\
	r_{k,i}=0 \text{ if }i<k\\
	r_{k,k}\geq q,\,\, \exists k \text{ with } k\neq 0,T_{l}\\
	\sum_{k} k r_{k,k}=n-1\\
		} } \binom{q+1}{\underline{r,i}}\underline{b,i}^{\underline{r,i}}$$

\vskip8pt

Here $\text{ord}_{E^{q+1}-1}\frac{E^{q+1}-1}{\pi}b_{n-1,n-1}\geq 1+R(n-1)=1+ R(n)$ by induction as $R(n-1)=R(qT_l)=q^{l+1}=R(T_{l+1})=R(n)$.

\vskip8pt

Also, for some $k\neq 0,T_l$ and $r_{k,k}>q$, $\text{ord}_{E^{q+1}-1}\underline{b,i}^{\underline{r,i}}\geq \sum_k r_{k,k}R(k)> R(\sum kr_{k,k})=R(qT_l)=R(T_{l+1})$ because of the following:

\vskip8pt

It is clear that $\sum r_{k,k}\sigma(k)<\sigma(qT_{l})=:\{n_r\}_{r\geq 0}$ and $n_r=0$ for all $r\geq$ except $n_l=q$. Then there finite many steps as in the proof of Lemma \ref{lem3}, $\sum r_{k,k}\sigma(k)<\{N^1_r\}<\{N^2_r\}<\dots<\{N^k_r\}<\sigma(qT_{l})$ with non-increasing $R'$ values. Since $n_r=0$ for $r\neq l$, we see that $N^k_0$ can only be $1$, and so $R'(\{N^k_r\})=1+R(qT_l)$.\vskip8pt

Finally, $\text{ord}_{E^{q+1}-1}b_{T_{l},T_{l}}^{q}b_0=qR(T_l)=R(T_{l+1})$ and hence $\text{ord}_{E^{q+1}-1}b_{T_{l+1},T_{l+1}}=R(T_{l+1})$. \end{proof}

\vskip8pt

\begin{rem}
The above proposition show that the lower bound $R(n)$ is actually sharp for infinitely many $n$.
\end{rem}

\vskip8pt

Now we consider the vanishing order of $Q_{n,s}$ at $x=1$.

\begin{prop}
\label{prop1}
\begin{enumerate}
	\item $\text{ord}_{x-1}Q_{n,s}\geq R\left(s-2\left\lfloor \frac{n-s}{q-1}\right\rfloor\right)$ for all $0\leq s \leq n$\;.
	\item When $s=2\left\lfloor \frac{n-s}{q-1}\right\rfloor$ and $n>0$, $\text{ord}_{x-1}Q_{n,s}\geq 1$\;.
\end{enumerate}

\end{prop}

\begin{proof}
We will prove (1) and (2) by induction on $n$.\vskip8pt

First look at the case $1\leq n \leq q-2$. By definition,

\[b_{1,1}=\frac{E^{q+1}-1}{\pi}E\]
\vskip8pt

 and $b_{1,0}=0$. For $2\leq n \leq q-2$ and $0<s\leq n$,

$$b_{n,s} \; = \; \frac{E^{q+1}-1}{\pi}b_{n-1,s-1}  +\frac{q}{\pi}E^{q+1}b_{n-1,s} +\pi^{-1}E \sum_{\substack{|\underline{r,i}|=q+1\\
	r_{k,i}\leq q-1,\,\, \forall (k,i) \\
	\sum_{k,i} k r_{k,i}=n-1\\
	\sum_{k,i} i r_{k,i}=s\\
	} } \binom{q+1}{\underline{r,i}}\underline{b,i}^{\underline{r,i}} \;.$$
	
\vskip8pt

And $ b_{n,0}=0$ for $2\leq n\leq q-2$. Notice that $0=2\left\lfloor \frac{n-0}{q-1}\right\rfloor$ iff $0\leq n\leq q-2$.\vskip8pt

By induction, we can show $\text{ord}_{x-1}Q_{n,s}\geq s$ when $0\leq s \leq n\leq q-2$.\vskip8pt

Now suppose the statement for $0\leq s\leq n<N$ where $N\geq q-1$.\vskip8pt
We will check each term of $b_{N,s}$ in equation \ref{eq4}.\vskip8pt

For terms $\left(E^{1+q^{\left\lfloor \log_q(m)\right\rfloor}}-1\right)\binom{l}{\underline{r,i}}\underline{b,i}^{\underline{r,i}}$ with $l<m$, $|\underline{r,i}|=l$, $\sum_{k,i} kr_{k,i}=N-\frac{m+l-1}{q+1}$ and $\sum_{k,i}ir_{k,i}=s+\nu(c_md_l)$,

\[\begin{aligned}
\text{ord}_{E^{q+1}-1} & \geq 1+\sum r_{k,i}R\left(i-2\left\lfloor \frac{k-i}{q-1}\right\rfloor\right)\\
 & \geq 1+R\left(\sum ir_{k,i}-2\sum r_{k,i} \left\lfloor \frac{k-i}{q-1}\right\rfloor\right)\\
 & \geq 1+R\left(s+\nu(c_md_l)-2\left\lfloor \frac{N-\frac{m+l-1}{q+1}-s-\nu(c_md_l)}{q-1}\right\rfloor\right)\;.\\
\end{aligned}\]
\vskip8pt

Observe that if $m+l>q+2$,

\[s+\nu(c_md_l)-2\left\lfloor \frac{N-\frac{m+l-1}{q+1}-s-\nu(c_md_l)}{q-1}\right\rfloor\geq s-2\left\lfloor \frac{N-s}{q-1}\right\rfloor\;.\]
\vskip8pt

To show this, we can separate in into two cases: $q+2<m+l<q^3+2$ and $m+l\geq q^3+2$. For the first case, the only non zero $d_lc_m$ is at $m+l=q^2,q^2+q+1,q^2+2q+2$. And the inequality can be check directly. For the second case, notice that $-\left\lfloor \log_q(m+l)\right\rfloor\leq \nu(c_md_l)\leq 0$ and hence
$\frac{m+l-1}{q+1}\geq (q-1)q^{\left\lfloor \log_q(m+l)\right\rfloor-2}-(\nu(c_md_l))$ and $2q^{\left\lfloor \log_q(m+l)\right\rfloor-2}+ \nu(c_md_l)> 0$. So we actually have $s+\nu(c_md_l)-2\left\lfloor \frac{N-\frac{m+l-1}{q+1}-s-\nu(c_md_l)}{q-1}\right\rfloor> s-2\left\lfloor \frac{N-s}{q-1}\right\rfloor$ in the second case.\vskip8pt

Hence,

\[\text{ord}_{x-1}\geq \begin{cases}
1+R\left(s-1-2\left\lfloor \frac{N-s}{q-1}\right\rfloor\right), & \text{ if }m+l=q+2\\
1+R\left(s-2\left\lfloor \frac{N-s}{q-1}\right\rfloor\right), & \text{ if }m+l>q+2\\
\end{cases}\] 
\vskip8pt

which is at least $R\left(s-2\left\lfloor \frac{N-s}{q-1}\right\rfloor\right)$. And for any $N$ and $s$, it is at least 1.\vskip8pt
Similarly for the respective term when $m<l$.

\vskip8pt

For terms $\binom{l}{\underline{r,i}}\underline{b,i}^{\underline{r,i}}$ with $l<m$, $|\underline{r,i}|=l$ or $m$, $\sum_{k,i} kr_{k,i}=N-\frac{m+l-1}{q+1}$ and $\sum_{k,i}ir_{k,i}=s+\nu(c_md_l)+1$,

\[\begin{aligned}
\text{ord}_{E^{q+1}-1} & \geq R\left(s+\nu(c_md_l)+1-2\left\lfloor \frac{N-\frac{m+l-1}{q+1}-s-\nu(c_md_l)-1}{q-1}\right\rfloor\right)\\
\\
& \geq \begin{cases}
R\left(s-2\left\lfloor \frac{N-s}{q-1}\right\rfloor\right), & \text{ if }m+l=q+2\\
R\left(s+1-2\left\lfloor \frac{N-s-1}{q-1}\right\rfloor\right), & \text{ if }m+l>q+2\\
\\
\end{cases}\\
& \geq R\left(s-2\left\lfloor \frac{N-s}{q-1}\right\rfloor\right)\;.\\
\end{aligned}\]
\vskip8pt

Suppose $s-2\left\lfloor \frac{N-s}{q-1}\right\rfloor= 0$ and $s\geq 2$, then there must exist $k\geq i$ with $k>0$ s.t. $r_{k,i}>0$ and $i-2\left\lfloor \frac{k-i}{q-1}\right\rfloor\geq 0$. Otherwise, we have all $k,i$ with $r_{k,i}>0$, $i+1-2\left\lfloor \frac{k-i}{q-1}\right\rfloor \leq 0 $ or $k=0=s$. If $m+l<q^3$, then $s+\nu(c_md_l)+1>0$ and so $r_{0,0}<l$. Therefore,

\[0>\sum r_{k,i}\left(i-2\left\lfloor \frac{k-i}{q-1}\right\rfloor\right)\geq s+\nu(c_md_l)+1-2\left\lfloor \frac{N-\frac{m+l-1}{q+1}-s-\nu(c_md_l)-1}{q-1}\right\rfloor\]\[ \geq s-2\left\lfloor \frac{N-s}{q-1}\right\rfloor\geq 0\]
\vskip8pt

 and leads to a contraction. If $m+l\geq q^3$, then

\[0\geq \sum r_{k,i}\left(i-2\left\lfloor \frac{k-i}{q-1}\right\rfloor\right)\geq s+\nu(c_md_l)+1-2\left\lfloor \frac{N-\frac{m+l-1}{q+1}-s-\nu(c_md_l)-1}{q-1}\right\rfloor\]\[ > s-2\left\lfloor \frac{N-s}{q-1}\right\rfloor\geq 0\]
\vskip8pt

and again leads to contradiction. So $E^{q+1}-1$ divides some $b_{k,i}$ with $r_{k,i}>0$ and thus divides $\underline{b,i}^{\underline{r,i}}$.\vskip8pt

For terms $\binom{l}{\underline{r,i}}\underline{b,i}^{\underline{r,i}}$ with $m<l$, $|\underline{r,i}|=l$, $r_{0,0}<q^{\left\lfloor \log_q(l)\right\rfloor}$, $p\nmid\binom{l}{\underline{r,i}}$, $\sum_{k,i} kr_{k,i}=N-\frac{m+l-1}{q+1}$ and $\sum_{k,i}ir_{k,i}=s+\nu(c_md_l)$. In particular, there exists $(k_0,i_0)\neq (0,0)$ s.t. $r_{k_0,i_0}\geq q^{\left\lfloor \log_q(l)\right\rfloor}$.\vskip8pt

Case 1)\;\; $i_o-2\left\lfloor \frac{k_0-i_0}{q-1}\right\rfloor>0$. Then

\[\begin{aligned}
\text{ord}_{E^{q+1}-1} & \geq \sum r_{k,i}R\left(i-2\left\lfloor \frac{k-i}{q-1}\right\rfloor\right) &\\
\\
 & \geq R\left(1+\sum ir_{k,i}-2\sum r_{k,i} \left\lfloor \frac{k-i}{q-1}\right\rfloor\right) &\text{ by lemma }\ref{lem3}(3) \\
\\
 & \geq R\left(1+s+\nu(c_md_l)-2\left\lfloor \frac{N-\frac{m+l-1}{q+1}-s-\nu(c_md_l)}{q-1}\right\rfloor\right) & \\
\\
& \geq \begin{cases}
R\left(s-2\left\lfloor \frac{N-s}{q-1}\right\rfloor\right), & \text{ if }m+l=q+2\\
\\
R\left(1+s-2\left\lfloor \frac{N-s}{q-1}\right\rfloor\right), & \text{ if }m+l>q+2\\
\end{cases} \;.& \\
\end{aligned}\]
\vskip8pt

And $b_{k_0,i_0}$ is divisible by $E^{q+1}-1$ implies the same holds for $\underline{b,i}^{\underline{r,i}}$.\vskip8pt

Case 2)\;\; $i_o-2\left\lfloor \frac{k_0-i_0}{q-1}\right\rfloor=0$. Then $b_{k_0,i_0}$ is divisible by $E^{q+1}-1$ and so is $\underline{b,i}^{\underline{r,i}}$.

\[\begin{aligned}
\text{ord}_{E^{q+1}-1} & \geq r_{k_0,i_0}+\sum r_{k,i}R\left(i-2\left\lfloor \frac{k-i}{q-1}\right\rfloor\right) \\
 & \geq R\left(r_{k_0,i_0}+s+\nu(c_md_l)-2\left\lfloor \frac{N-\frac{m+l-1}{q+1}-s-\nu(c_md_l)}{q-1}\right\rfloor\right)  \\
& \geq \begin{cases}
R\left(r_{k_0,i_0}-1+s-2\left\lfloor \frac{N-s}{q-1}\right\rfloor\right), & \text{ if }m+l=q+2\\
R\left(r_{k_0,i_0}+s-2\left\lfloor \frac{N-s}{q-1}\right\rfloor\right), &  \text{ if }m+l>q+2\;.\\
\end{cases} \\
\end{aligned}\]
\vskip8pt

Case 3)\;\; $i_o-2\left\lfloor \frac{k_0-i_0}{q-1}\right\rfloor \leq -1 $.
Then

\[\begin{aligned}
\sum_{(k,i)\neq (k_0,i_0)}r_{k,i}(i-2\left\lfloor \frac{k-i}{q-1}\right\rfloor) &\geq r_{k_o,i_0}+\sum r_{k,i}(i-2\left\lfloor \frac{k-i}{q-1}\right\rfloor)\\
& \geq r_{k_o,i_0}+s+\nu(c_md_l)-2\left\lfloor \frac{N-\frac{m+l-1}{q+1}-s-\nu(c_md_l)}{q-1}\right\rfloor\\
& \geq \begin{cases}
r_{k_0,i_0}-1+s-2\left\lfloor \frac{N-s}{q-1}\right\rfloor, & \text{ if }m+l=q+2\\
r_{k_0,i_0}+s-2\left\lfloor \frac{N-s}{q-1}\right\rfloor, &  \text{ if }m+l>q+2\;.\\
\end{cases} \\
\end{aligned}\]
\vskip8pt

So

\[\begin{aligned}
\text{ord}_{E^{q+1}-1} & \geq \sum_{(k,i)\neq (k_0,i_0)} r_{k,i}R\left(i-2\left\lfloor \frac{k-i}{q-1}\right\rfloor\right) \\
 & \geq R\left(s-2\left\lfloor \frac{N-s}{q-1}\right\rfloor\right)\;.\\
\end{aligned}\]
\vskip8pt

Suppose $s-2\left\lfloor \frac{N-s}{q-1}\right\rfloor=0$, then $\sum_{(k,i)\neq (k_0,i_0)}r_{k,i}(i-2\left\lfloor \frac{k-i}{q-1}\right\rfloor) >0$. Hence there is $k\geq i$ with $r_{k,i}>0 $ and $i-2\left\lfloor \frac{k-i}{q-1}\right\rfloor>0$. Hence $E^{q+1}|\underline{b,i}^{\underline{r,i}}$.\vskip8pt

The argument is similar for the respective terms with $l<m$. This completes the proof of the proposition. \end{proof}

\vskip8pt

\subsection{Radius of convergence for the group action on $\Delta_0$}

As a consequence of \ref{prop1} we obtain the following

\begin{thm}\label{main1}
\begin{enumerate}
	\item Suppose $|x-1|\leq |\pi|$. Then $\left|\pi^{-n}Q_n\right|\leq |\pi|^{\frac{-2}{q}n}$. Hence $|\alpha-1|\leq |\pi|$ implies that $\left|b_{n}(\frac{\bar{\alpha}}{\alpha})\right|\leq |\pi|^{\frac{-2}{q}n}$.
	
	\vskip8pt
	
	\item Suppose $|\pi|^{\frac{q}{q+1}}\leq r<1$ and $|x-1|\leq r$. Then $|Q_n(x)|\leq r^{\frac{n(q-1)}{q}}$. In particular, the action of
	
$$\left\{\left(\begin{array}{cc}
	\alpha & 0\\
	0 & \bar{\alpha}\\
	\end{array}\right) \; \Big| \; \alpha \in 1+\pi \fro_{K_2}^* \right\} \sub \froDx = G$$
	
	\vskip8pt
	
on $\Delta_0 = \{|u|\leq |\pi|^{\frac{1}{q+1}}\}$ extends to a rigid-analytic action of the rigid-analytic group $\bbT_0^\circ$, cf. \ref{non_split_torus}, on $\Delta_0$.
\end{enumerate}
\end{thm}

\begin{proof}
Write $\pi^{-n}Q_n=\sum \pi^{-s}Q_{n,s}=\pi^{-n}\sum\pi^{n-s}Q_{n,s}$.

\vskip8pt

(1) Suppose $|x-1|\leq |\pi|$, then

\[\begin{aligned}
|\pi^{n-s}Q_{n,s}(x)| &\leq |\pi|^{(n-s)+R\left(s-2\left\lfloor \frac{n-s}{q-1}\right\rfloor\right)}\\
 & \leq |\pi|^{(n-s)+\frac{q-1}{q}\left(s-2 \frac{n-s}{q-1}\right)}\\
 & = |\pi|^{\frac{(q-2)n+2s}{q}}\\
 & \leq |\pi|^{\frac{q-2}{q}n}\;. \\
\end{aligned}\]

\vskip8pt

Hence $|\pi^{-n}Q_n|\leq |\pi|^{\frac{-2}{q}n}$.

\vskip8pt

(2) Suppose $|x-1|\leq r$ where $ |\pi|^{\frac{q}{q+1}}\leq r<1$. Then

\[\begin{aligned}
|\pi^{n-s}Q_{n,s}(x)| &\leq r^{\frac{q+1}{q}(n-s)+R\left(s-2\left\lfloor \frac{n-s}{q-1}\right\rfloor\right)}\\
 & \leq r^{\frac{q+1}{q}(n-s)+\frac{q-1}{q}\left(s-2\frac{n-s}{q-1}\right)}\\
 & = r^{\frac{(q-1)n}{q}}\;. \\
\end{aligned}\]
\vskip8pt

In particular, when $|\alpha-1|\leq r$ and $|u|\leq |\pi|^{\frac{1}{q+1}}$,

\[|b_n(\alpha)u^{1+n(q+1)}|\leq r^{\frac{q-1}{q}n}|\pi|\rightarrow 0\]
\vskip8pt

as $n\rightarrow \infty$. \end{proof}


\section{Analyticity on critical discs of larger radius}\label{larger_radii}

In this section, put  $r_s = |\pi|^{\frac{1}{(1+q)q^s}}$ for $s \in \Z_{\ge 0}$, which we sometimes call a {\it critical radius}. The function $\phi_1$ vanishes at $u=0$, and all its other zeros are located on the annuli

$$\cA_s = \{u \in X \; | \; |u| = r_s \;\} \;, $$
\vskip8pt

for {\it odd} $s = 1,3,5, \ldots \,$, and $\phi_1(u)$ has precisely $(1+q)q^s$ zeros on $\cA_s$ for odd $s$. The zeros of $\phi_0(u)$ are located on the annuli $\cA_s$ for {\it even} $s = 0,2, 4, \ldots \,$, and $\phi_0(u)$ has precisely $(1+q)q^s$ zeros on $\cA_s$ for even $s$. Let

$$\Delta_s = \{u \in X \; | \; |u| \le r_s \;\} \;, $$
\vskip8pt

be the disc of critical radius $r_s$ centered at zero.

\subsection{Estimates for the action of the Lie algebra}

\begin{lem}\label{supremums}
\begin{enumerate}
\item Let $s \ge 0$ be even. Then

$$||\phi_0||_{\Delta_s} = |\pi|^{- \frac{s}{2}+\frac{q^s-1}{(q^2-1)q^s}} \;, \hskip10pt ||\phi_1||_{\Delta_s} = |\pi|^{- \frac{s}{2}+\frac{q^{s+1}-1}{(q^2-1)q^s}} \;.$$
\vskip8pt

\item Let $s \ge 1$ be odd. Then

$$||\phi_0||_{\Delta_s} = |\pi|^{- \frac{s+1}{2}+\frac{q^{s+1}-1}{(q^2-1)q^s}} \;, \hskip10pt ||\phi_1||_{\Delta_s} = |\pi|^{- \frac{s-1}{2}+\frac{q^{s}-1}{(q^2-1)q^s}} \;.$$
\vskip8pt

\item For all $s \ge 0$ one has

$$||\phi_0\phi_1||_{\Delta_s} = |\pi|^{-s+\frac{1}{q-1} - \frac{2}{(q^2-1)q^s}} \;.$$
\vskip8pt

\end{enumerate}
\end{lem}
\begin{proof}
(1) It is not hard to check that $\pi^{-\frac{s}{2}}u^{1+q+q^2+\dots+q^{s-1}}$ and $\pi^{-\frac{s+2}{2}}u^{1+q+q^2+\dots+q^{s+1}}$ are the dominating terms of $\phi_0$ when $|u|=r_s$ for $s>0$. $1$ and $\pi^{-1} u^{1+q}$ are the dominating terms of $\phi_0$ when $s=0$. Therefore,

\[||\phi_0||_{\Delta_s}=|\pi|^{-\frac{s}{2}}r_s^{\frac{q^{s}-1}{q-1}}=|\pi|^{- \frac{s}{2}+\frac{q^s-1}{(q^2-1)q^s}}\;.\]
\vskip8pt

Similarly, $\pi^{-\frac{s}{2}}u^{1+q+q^2+\dots+q^{s}}$ is the dominating term of $\phi_1$. Hence

\[||\phi_1||_{\Delta_s}=|\pi|^{-\frac{s}{2}}r_s^{\frac{q^{s+1}-1}{q-1}}=|\pi|^{- \frac{s}{2}+\frac{q^{s+1}-1}{(q^2-1)q^s}}\;.\]
\vskip8pt

The treatment for part (2) is the same and part (3) follows from (1), (2). \end{proof}

\vskip8pt

\begin{para}
{\it The action of the Lie algebra.} Let $\zeta\in \fro_{K_2}^*$ be such that $\bar{\zeta}=-\zeta$. Let $\bbG$ be the group scheme over $\fro = \fro_K$ whose $\fro$-valued points are $\froDx$, cf. section \ref{intro}. Denote by $\frg$ the relative Lie algebra of $\bbG$ over $\fro$. Consider the following $\fro$-basis of $\frg$:

$$\frx_1 = \left(\begin{array}{cc} 1 & 0 \\
0 &  1 \end{array}\right) \;, \;\; \frx_2 = \left(\begin{array}{cc} \zeta & 0 \\
0 & -\zeta \end{array}\right) \;, \;\; \fry_1 = \left(\begin{array}{cc} 0 & \pi \\
1 & 0 \end{array}\right) \;, \;\; \fry_2 = \left(\begin{array}{cc}  0 & -\pi\zeta \\
 \zeta & 0 \end{array}\right) \;.$$

\vskip8pt

We note that

$$[\frx_2,\fry_1] = -2\fry_2 \;,\;\; [\frx_2,\fry_2] = -2\zeta^2\fry_1 \;, \;\; [\fry_1,\fry_2] = 2\pi\frx_2 \;.$$

\vskip8pt

We write elements of $\bbP^1$ as $[x_0:x_1]$ and put $w = \frac{x_1}{x_0}$. Elements $\frz$ in the Lie algebra $\frg$ act on rational functions $f = f(w)$ on $\bbP^1$ as follows:

$$(\frz.f)(w) = \frac{d}{dt} f(e^{t\frz}.w)\Big|_{t=0} \;.$$

\vskip8pt

Note that

$$\exp(t\fry_1) = \left(\begin{array}{cc} \cosh(t\sqrt{\pi}) & \sqrt{\pi} \sinh(t\sqrt{\pi}) \\
\sinh(t\sqrt{\pi})/\sqrt{\pi} &  \cosh(t\sqrt{\pi}) \end{array}\right) \;, $$

\vskip8pt

and

$$\exp(t\fry_2) = \left(\begin{array}{cc} \cos(t\zeta\sqrt{\pi}) & -\sqrt{\pi}\sin(t\zeta\sqrt{\pi})\\
\sin(t\zeta\sqrt{\pi})/\sqrt{\pi} &  \cos(t\zeta\sqrt{\pi}) \end{array}\right) \;. $$

\vskip8pt

According to the formula \ref{action} for the group action we compute

$$\begin{array}{rcl}(\frx_2.f)(w) & = & \frac{d}{dt} f(e^{-2t\zeta} w)\Big|_{t=0}  = -2\zeta wf'(w) \;,\\
&&\\
(\fry_1.f)(w) & = & \frac{d}{dt} f\left(\frac{\cosh(t\sqrt{\pi})w+\sqrt{\pi}\sinh(t\sqrt{\pi})}{\sinh(t\sqrt{\pi})/\sqrt{\pi} \cdot w + \cosh(t\sqrt{\pi})}\right)\Big|_{t=0} \\
&&\\
&=& \frac{\sinh(t\sqrt{\pi})\sqrt{\pi}w+\pi\cosh(t\sqrt{\pi})- w(\cosh(t\sqrt{\pi})w + \sqrt{\pi}\sinh(t\sqrt{\pi}))}{(\sinh(t\sqrt{\pi})/\sqrt{\pi} \cdot w + \cosh(t\sqrt{\pi}))^2}\Big|_{t=0} \cdot f'(w)\\
&&\\
&=& (\pi-w^2)f'(w) \;,\\
&&\\
(\fry_2.f)(w) & = & \frac{d}{dt} f\left(\frac{\cos(t\zeta\sqrt{\pi})w-\sqrt{\pi}\sin(t\zeta\sqrt{\pi})}{\sin(t\zeta\sqrt{\pi})/\sqrt{\pi} \cdot w + \cos(t\zeta\sqrt{\pi})}\right)\Big|_{t=0} \\
&&\\
&=& \frac{-\sin(t\zeta\sqrt{\pi})\zeta\sqrt{\pi}w-\zeta\pi\cos(t\zeta\sqrt{\pi})- w(\zeta\cos(t\sqrt{\pi})w + -\zeta\sqrt{\pi}\sin(t\zeta\sqrt{\pi}))}{(\sin(t\zeta\sqrt{\pi})/\sqrt{\pi} \cdot w + \cos(t\zeta\sqrt{\pi}))^2}\Big|_{t=0} \cdot f'(w)\\
&&\\
&=& -\zeta(\pi+w^2)f'(w) \;.
\end{array}$$

\vskip8pt

Therefore, under the map from $\frg_\fro$ to the ring of differential operators on $P$ we have:

\begin{numequation}\label{diff_action}\frx_2 \mapsto -2\zeta w\partial_w \;, \hskip10pt \fry_1 \mapsto (\pi-w^2)\partial_w \;, \hskip10pt \fry_2 \mapsto -\zeta (\pi+w^2)\partial_w \;.
\end{numequation}

Setting $w = \frac{\phi_1(u)}{\phi_0(u)}$, we find

$$dw = \frac{\phi_1'\phi_0-\phi_1\phi_0'}{\phi_0^2}du = \frac{\vep}{\phi_0^2}du \;,$$

\vskip8pt

where $\vep = \phi_1'\phi_0-\phi_1\phi_0'$. As $\partial_w$ is dual to $dw$ we get thus

$$1 = \langle \partial_w, dw \rangle = \frac{\vep}{\phi_0^2}\langle \partial_w, du \rangle \;$$

\vskip8pt

and hence $\partial_w = \frac{\phi_0^2}{\vep} \partial_u$. From \ref{diff_action} we then deduce

\begin{numequation}\label{diff_action_u}\frx_2 \mapsto -2\zeta \frac{\phi_0\phi_1}{\vep}\partial_u \;, \hskip10pt \fry_1 \mapsto \frac{\pi\phi_0^2-\phi_1^2}{\vep}\partial_u \;, \hskip10pt \fry_2 \mapsto -\zeta \frac{\pi\phi_0^2+\phi_1^2}{\vep}\partial_u \;.
\end{numequation}

Let now $\pi^{\frac{1}{(q+1)q^s}}$ be any element of absolute value equal to $|\pi|^{\frac{1}{(q+1)q^s}}$, and put $u_s = \pi^{-\frac{1}{(q+1)q^s}}u$, which is a coordinate function on $\Delta_s$ with supremum norm $1$. Then $\partial_u = \pi^{-\frac{1}{(q+1)q^s}}\partial_{u_s}$ and the formulas in \ref{diff_action_u} become

\begin{numequation}\label{diff_action_u_s}\begin{array}{rcl}\frx_2 & \mapsto & -2\zeta \frac{\phi_0\phi_1}{\vep}\pi^{-\frac{1}{(q+1)q^s}}\partial_{u_s}\;, \\
&&\\
\fry_1 & \mapsto & \frac{\pi\phi_0^2-\phi_1^2}{\vep}\pi^{-\frac{1}{(q+1)q^s}}\partial_{u_s} \;, \\
&&\\
\fry_2 & \mapsto & -\zeta \frac{\pi\phi_0^2+\phi_1^2}{\vep}\pi^{-\frac{1}{(q+1)q^s}}\partial_{u_s} \;.
\end{array}
\end{numequation}

\end{para}

\begin{prop}\label{operator_estimates}
\begin{enumerate}
\item Let $s \ge 0$ be even. Then

$$||\pi^{s+1}\phi_0^2\pi^{-\frac{1}{(q+1)q^s}}||_{\Delta_s} = |\pi|^{1+\frac{2q^s-2}{(q^2-1)q^s}-\frac{1}{(q+1)q^s}} = |\pi|^{1+\frac{2}{q^2-1}-\frac{1}{(q-1)q^s}} = |\pi|^{\frac{q^2+1}{q^2-1}-\frac{1}{(q-1)q^s}}\;,$$

\vskip8pt

and

$$||\pi^s\phi_1^2\pi^{-\frac{1}{(q+1)q^s}}||_{\Delta_s} = |\pi|^{\frac{2q^{s+1}-2}{(q^2-1)q^s}-\frac{1}{(q+1)q^s}} = |\pi|^{\frac{2q}{q^2-1}-\frac{1}{(q-1)q^s}}\;.$$

\vskip8pt

In particular, when $s=0$:

$$||\pi^{s+1}\phi_0^2\pi^{-\frac{1}{(q+1)q^s}}||_{\Delta_0} =  |\pi|^{\frac{q}{q+1}} \;, \hskip20pt ||\pi^s\phi_1^2\pi^{-\frac{1}{(q+1)q^s}}||_{\Delta_0} = |\pi|^{\frac{1}{q+1}}\;.$$

\vskip8pt

\item Let $s \ge 1$ be odd. Then

$$||\pi^{s+1}\phi_0^2\pi^{-\frac{1}{(q+1)q^s}}||_{\Delta_s} = |\pi|^{\frac{2q^{s+1}-2}{(q^2-1)q^s}-\frac{1}{(q+1)q^s}} = |\pi|^{\frac{2q}{q^2-1}-\frac{1}{(q-1)q^s}} \;, $$

\vskip8pt

and

$$||\pi^s\phi_1^2\pi^{-\frac{1}{(q+1)q^s}}||_{\Delta_s} = |\pi|^{1+\frac{2q^{s}-2}{(q^2-1)q^s}-\frac{1}{(q+1)q^s}} =  |\pi|^{1+\frac{2}{q^2-1}-\frac{1}{(q-1)q^s}} = |\pi|^{\frac{q^2+1}{q^2-1}-\frac{1}{(q-1)q^s}} \;.$$

\vskip8pt

\item For all $s \ge 0$ one has

$$||\pi^s\phi_0\phi_1\pi^{-\frac{1}{(q+1)q^s}}||_{\Delta_s} = |\pi|^{\frac{1}{q-1} - \frac{2}{(q^2-1)q^s}-\frac{1}{(q+1)q^s}} = |\pi|^{\frac{1}{q-1}-\frac{1}{(q-1)q^s}} \;.$$

\vskip8pt
\end{enumerate}
\end{prop}

\subsection{Groups acting analytically on larger critical discs}

As an immediate consequence of the previous proposition \ref{operator_estimates} we obtain the following result.

\begin{thm}\label{main2} Let $K = \Qp$ (hence $q=p$ and $\pi = p$), $K_s = \Qp(p^{\frac{1}{(p^1-1)p^s}})$, and put

$$\frh_s = \fro \cdot p^s \frx_1 \oplus \fro \cdot p^{s+ \frac{1}{(p-1)p^s}} \frx_2 \oplus \fro \cdot  p^{s-\frac{1}{p+1}+\frac{1}{(p-1)p^s}}\fry_1 \oplus \fro \cdot
p^{s-\frac{1}{p+1}+\frac{1}{(p-1)p^s}}\fry_2 \;. $$

\vskip8pt

This is a Lie algebra over the ring of integers $\fro_{K_s}$ in $K_s$. There is a group scheme $\bbH_s$ over $\fro_{K_s}$ with Lie algebra $\frh_s$. Denote by $\widehat{\bbH}_s^\circ$ the completion of this group scheme along the unit section, and let $\bbH_s^\circ$ be the associated rigid-analytic group. Then $\bbH_s^\circ$ acts analytically on $\Delta_s$.
\end{thm}

\vskip8pt

\begin{rem}
Suppose $K = \Qp$ and consider the case $s=0$. Then theorem \ref{main1} (2) implies that, in the formula for $\frh_0$ above, we can replace $\fro \cdot p^{ \frac{1}{(p-1)}} \frx_2$ by $\fro \cdot \frx_2$. Hence, for $s=0$, we can replace the Lie algebra $\frh_0$ in the theorem above by

$$\frh_0' = \fro \cdot  \frx_1 \oplus \fro \cdot  \frx_2 \oplus \fro \cdot  p^{-\frac{1}{p+1}+\frac{1}{(p-1)}}\fry_1 \oplus \fro \cdot
p^{-\frac{1}{p+1}+\frac{1}{(p-1)}}\fry_2 \;. $$
\end{rem}

\vskip8pt

\begin{para}
Let $s\geq 0$. Suppose $u_0\in \Delta_s$ and let $B^-(u_0,r)$ be the largest wide open disc such that $\Phi$ is injective on $B^-(u_0,r)$. In \cite{Lo} we call $B^-(u_0,r)$ a disc of injectivity around $u_0$, and we have shown that $r = |\pi u_0^{-2}|^{\frac{1}{q-1}}$. In \cite[sec. 3]{Lo} we describe the image of $B^-(u_0,r_0)$ under $\Phi$ which is again a wide open disc and whose radius we determine.

\vskip8pt

Suppose $g=\left(
\begin{array}{cc}
	\alpha &  \pi\bar{\beta}\\
	\beta & \bar{\alpha} \\
\end{array}
\right)\in G$. Then $g\cdot \Phi(u_0)\in \Phi(B^-(u_0,r_0))$ for all $u_0\in \Delta_s$ if and only if $|\alpha-\bar{\alpha}|<|\pi|^{s}$ and $|\beta|<|\pi|^{s-\frac{1}{q+1}}$. Therefore, we would expect that we could actually replace the Lie algebra $\frh_s$ in theorem \ref{main2} by
the larger Lie algebra

$$\frg_s = \fro \cdot p^s \frx_1 \oplus \fro \cdot p^{s} \frx_2 \oplus \fro \cdot  p^{s-\frac{1}{p+1}}\fry_1 \oplus \fro \cdot p^{s-\frac{1}{p+1}}\fry_2$$

\vskip8pt

which differs from $\frh_s$ by the factor $p^{\frac{1}{(p-1)p^s}}$ in front of the generators $\frx_2$, $\fry_1$, and $\fry_2$.
\end{para}


\bibliographystyle{plain}

\bibliography{ref}

\begin{thebibliography}{1}

\bibitem{Drinfeld74}
V.~G. Drinfeld.
\newblock Elliptic modules.
\newblock {\em Mat. Sb. (N.S.)}, 94(136):594--627, 656, 1974.

\bibitem{Emerton}
M.~Emerton.
\newblock Locally analytic vectors in representations of locally $p$-adic
  analytic groups. {P}reprint. {T}o appear in: {M}emoirs of the {A}{M}{S}.

\bibitem{GH}
M.~J. Hopkins and B.~H. Gross.
\newblock Equivariant vector bundles on the {L}ubin-{T}ate moduli space.
\newblock In {\em Topology and representation theory ({E}vanston, {IL}, 1992)},
  volume 158 of {\em Contemp. Math.}, pages 23--88. Amer. Math. Soc.,
  Providence, RI, 1994.

\bibitem{Koh}
J.~Kohlhaase.
\newblock On the {I}wasawa theory of the {L}ubin-{T}ate moduli space.
\newblock {\em Compos. Math.}, 149(5):793--839, 2013.

\bibitem{KohlhaaseDef}
J.~Kohlhaase.
\newblock {I}wasawa {M}odules {A}rising from {D}eformation {S}paces of
  {$p$}-{D}ivisible {F}ormal {G}roup {L}aws.
\newblock {\em {I}wasawa {T}heory 2012. {C}ontributions in {M}athematical and
  {C}omputational {S}ciences}, 7:291--316, 2014.

\bibitem{Lo}
C.~Y. Lo.
\newblock {D}omains of {I}njectivity for the {G}ross-{H}opkins {P}eriod {M}ap.
  {P}reprint. http://arxiv.org/abs/1312.0034.

\end{thebibliography}

\end{document}